\documentclass[11pt]{amsart}
\usepackage{graphicx} 
\usepackage{amsmath, amssymb, amsthm}
\usepackage{adjustbox}

\usepackage{amsmath}
\usepackage{amsfonts}
\usepackage{amssymb}
\usepackage{amscd}
\usepackage{color}
\usepackage{xcolor}
\usepackage{hyperref}
\usepackage{mathrsfs}
\usepackage{eucal}
\usepackage{upgreek}
\usepackage[makeroom]{cancel}
\usepackage[normalem]{ulem}
\usepackage{array}
\usepackage{verbatim}
\usepackage{mathtools}
\usepackage{stmaryrd}
\usepackage[inline]{enumitem}
\usepackage{hyperref, cleveref}
\hypersetup{colorlinks=true}
\usepackage{xy}
\xyoption{all}
\usepackage{tikz-cd}

\newcommand{\nc}{\newcommand}

\nc{\md}{\operatorname{-}}

\renewcommand{\AA}{{\mathbb{A}}}
\nc{\CC}{{\mathbb{C}}}
\nc{\DD}{{\mathbb{D}}}
\nc{\LL}{{\mathbb{L}}}
\nc{\FF}{{\mathbb{F}}}
\nc{\RR}{{\mathbb{R}}}
\renewcommand{\P}{{\mathbb{P}}}
\nc{\OO}{{\mathbb{O}}}

\nc{\QQ}{{\mathbb{Q}}}
\nc{\ZZ}{{\mathbb{Z}}}
\nc{\Z}{{\mathbb{Z}}}

\nc{\cA}{{\mathcal{A}}}
\nc{\cB}{{\mathcal{B}}}
\nc{\cBstd}{{\cB^{\rm std}}}
\nc{\cC}{{\mathcal{C}}}
\nc{\cD}{{\mathcal{D}}}
\nc{\cE}{{\mathcal{E}}}
\nc{\cF}{{\mathcal{F}}}
\nc{\cG}{{\mathcal{G}}}
\nc{\cH}{{\mathcal{H}}}
\nc{\cI}{{\mathcal{I}}}
\nc{\cJ}{{\mathcal{J}}}
\nc{\cK}{{\mathcal{K}}}
\nc{\cKstd}{{\cK^{\rm std}}}
\nc{\cL}{{\mathcal{L}}}
\nc{\cM}{{\mathcal{M}}}
\nc{\cN}{{\mathcal{N}}}
\nc{\cO}{{\mathcal{O}}}
\nc{\cP}{{\mathcal{P}}}
\nc{\cQ}{{\mathcal{Q}}}
\nc{\cR}{{\mathcal{R}}}
\nc{\cS}{{\mathcal{S}}}
\nc{\cT}{{\mathcal{T}}}
\nc{\cU}{{\mathcal{U}}}
\nc{\cV}{{\mathcal{V}}}
\nc{\cW}{{\mathcal{W}}}
\nc{\cX}{{\mathcal{X}}}
\nc{\cY}{{\mathcal{Y}}}
\nc{\cZ}{{\mathcal{Z}}}

\nc{\rc}{{\mathrm{c}}}
\nc{\rd}{{\mathrm{d}}}
\nc{\rf}{{\mathrm{f}}}
\nc{\rh}{{\mathrm{h}}}
\nc{\rrm}{{\mathrm{m}}}
\nc{\rs}{{\mathrm{s}}}
\nc{\rch}{{\mathrm{ch}}}
\nc{\rtd}{{\mathrm{td}}}

\nc{\rA}{{\mathrm{A}}}
\nc{\rB}{{\mathrm{B}}}
\nc{\rC}{{\mathrm{C}}}
\nc{\rD}{{\mathrm{D}}}
\nc{\rE}{{\mathrm{E}}}
\nc{\rF}{{\mathrm{F}}}
\nc{\rG}{{\mathrm{G}}}
\nc{\rH}{{\mathrm{H}}}
\nc{\rI}{{\mathrm{I}}}
\nc{\rJ}{{\mathrm{J}}}
\nc{\rK}{{\mathrm{K}}}
\nc{\rL}{{\mathrm{L}}}
\nc{\rM}{{\mathrm{M}}}
\nc{\rN}{{\mathrm{N}}}
\nc{\rO}{{\mathrm{O}}}
\nc{\rP}{{\mathrm{P}}}
\nc{\rQ}{{\mathrm{Q}}}
\nc{\rR}{{\mathrm{R}}}
\nc{\rS}{{\mathrm{S}}}
\nc{\rT}{{\mathrm{T}}}
\nc{\rU}{{\mathrm{U}}}
\nc{\rV}{{\mathrm{V}}}
\nc{\rW}{{\mathrm{W}}}
\nc{\rX}{{\mathrm{X}}}
\nc{\rY}{{\mathrm{Y}}}
\nc{\rZ}{{\mathrm{Z}}}

\nc{\bA}{{\mathbf{A}}}
\nc{\bB}{{\mathbf{B}}}
\nc{\bC}{{\mathbf{C}}}
\nc{\bD}{{\mathbf{D}}}
\nc{\bE}{{\mathbf{E}}}
\nc{\bF}{{\mathbf{F}}}
\nc{\bG}{{\mathbf{G}}}
\nc{\bH}{{\mathbf{H}}}
\nc{\bI}{{\mathbf{I}}}
\nc{\bJ}{{\mathbf{J}}}
\nc{\bK}{{\mathbf{K}}}
\nc{\bL}{{\mathbf{L}}}
\nc{\bM}{{\mathbf{M}}}
\nc{\bN}{{\mathbf{N}}}
\nc{\bO}{{\mathbf{O}}}
\nc{\bP}{{\mathbf{P}}}
\nc{\bQ}{{\mathbf{Q}}}
\nc{\bR}{{\mathbf{R}}}
\nc{\bS}{{\mathbf{S}}}
\nc{\bT}{{\mathbf{T}}}
\nc{\bU}{{\mathbf{U}}}
\nc{\bV}{{\mathbf{V}}}
\nc{\bW}{{\mathbf{W}}}
\nc{\bX}{{\mathbf{X}}}
\nc{\bY}{{\mathbf{Y}}}
\nc{\bZ}{{\mathbf{Z}}}

\nc{\ba}{{\mathbf{a}}}
\nc{\bb}{{\mathbf{b}}}
\nc{\bc}{{\mathbf{c}}}
\nc{\bd}{{\mathbf{d}}}
\nc{\be}{{\mathbf{e}}}
\nc{\bg}{{\mathbf{g}}}
\nc{\bh}{{\mathbf{h}}}
\nc{\bi}{{\mathbf{i}}}
\nc{\bj}{{\mathbf{j}}}
\nc{\bk}{{\mathbf{k}}}
\nc{\bl}{{\mathbf{l}}}
\nc{\bm}{{\mathbf{m}}}
\nc{\bn}{{\mathbf{n}}}
\nc{\bo}{{\mathbf{o}}}
\nc{\bp}{{\mathbf{p}}}
\nc{\bq}{{\mathbf{q}}}
\nc{\br}{{\mathbf{r}}}
\nc{\bs}{{\mathbf{s}}}
\nc{\bt}{{\mathbf{t}}}
\nc{\bu}{{\mathbf{u}}}
\nc{\bv}{{\mathbf{v}}}
\nc{\bw}{{\mathbf{w}}}
\nc{\bx}{{\mathbf{x}}}
\nc{\by}{{\mathbf{y}}}
\nc{\bz}{{\mathbf{z}}}

\nc{\fA}{{\mathfrak{A}}}
\nc{\fB}{{\mathfrak{B}}}
\nc{\fC}{{\mathfrak{C}}}
\nc{\fD}{{\mathfrak{D}}}
\nc{\fE}{{\mathfrak{E}}}
\nc{\fF}{{\mathfrak{F}}}
\nc{\fG}{{\mathfrak{G}}}
\nc{\fH}{{\mathfrak{H}}}
\nc{\fI}{{\mathfrak{I}}}
\nc{\fJ}{{\mathfrak{J}}}
\nc{\fK}{{\mathfrak{K}}}
\nc{\fL}{{\mathfrak{L}}}
\nc{\fM}{{\mathfrak{M}}}
\nc{\fN}{{\mathfrak{N}}}
\nc{\fO}{{\mathfrak{O}}}
\nc{\fP}{{\mathfrak{P}}}
\nc{\fQ}{{\mathfrak{Q}}}
\nc{\fR}{{\mathfrak{R}}}
\nc{\fS}{{\mathfrak{S}}}
\nc{\fT}{{\mathfrak{T}}}
\nc{\fU}{{\mathfrak{U}}}
\nc{\fV}{{\mathfrak{V}}}
\nc{\fW}{{\mathfrak{W}}}
\nc{\fX}{{\mathfrak{X}}}
\nc{\fY}{{\mathfrak{Y}}}
\nc{\fZ}{{\mathfrak{Z}}}

\nc{\fa}{{\mathfrak{a}}}
\nc{\fb}{{\mathfrak{b}}}
\nc{\fc}{{\mathfrak{c}}}
\nc{\fd}{{\mathfrak{d}}}
\nc{\fe}{{\mathfrak{e}}}
\nc{\ff}{{\mathfrak{f}}}
\nc{\fg}{{\mathfrak{g}}}
\nc{\fh}{{\mathfrak{h}}}
\nc{\fj}{{\mathfrak{j}}}
\nc{\fk}{{\mathfrak{k}}}
\nc{\fl}{{\mathfrak{l}}}
\nc{\fm}{{\mathfrak{m}}}
\nc{\fn}{{\mathfrak{n}}}
\nc{\fo}{{\mathfrak{o}}}
\nc{\fp}{{\mathfrak{p}}}
\nc{\fq}{{\mathfrak{q}}}
\nc{\fr}{{\mathfrak{r}}}
\nc{\fs}{{\mathfrak{s}}}
\nc{\ft}{{\mathfrak{t}}}
\nc{\fu}{{\mathfrak{u}}}
\nc{\fv}{{\mathfrak{v}}}
\nc{\fw}{{\mathfrak{w}}}
\nc{\fx}{{\mathfrak{x}}}
\nc{\fy}{{\mathfrak{y}}}
\nc{\fz}{{\mathfrak{z}}}

\nc{\sA}{{\mathsf{A}}}
\nc{\sB}{{\mathsf{B}}}
\nc{\sC}{{\mathsf{C}}}
\nc{\sD}{{\mathsf{D}}}
\nc{\sE}{{\mathsf{E}}}
\nc{\sF}{{\mathsf{F}}}
\nc{\sG}{{\mathsf{G}}}
\nc{\sH}{{\mathsf{H}}}
\nc{\sI}{{\mathsf{I}}}
\nc{\sJ}{{\mathsf{J}}}
\nc{\sK}{{\mathsf{K}}}
\nc{\sL}{{\mathsf{L}}}
\nc{\sM}{{\mathsf{M}}}
\nc{\sN}{{\mathsf{N}}}
\nc{\sO}{{\mathsf{O}}}
\nc{\sP}{{\mathsf{P}}}
\nc{\sQ}{{\mathsf{Q}}}
\nc{\sR}{{\mathsf{R}}}
\nc{\sS}{{\mathsf{S}}}
\nc{\sT}{{\mathsf{T}}}
\nc{\sU}{{\mathsf{U}}}
\nc{\sV}{{\mathsf{V}}}
\nc{\sW}{{\mathsf{W}}}
\nc{\sX}{{\mathsf{X}}}
\nc{\sY}{{\mathsf{Y}}}
\nc{\sZ}{{\mathsf{Z}}}

\nc{\sa}{{\mathsf{a}}}
\nc{\sd}{{\mathsf{d}}}
\nc{\se}{{\mathsf{e}}}
\nc{\sg}{{\mathsf{g}}}
\nc{\sh}{{\mathsf{h}}}
\nc{\si}{{\mathsf{i}}}
\nc{\sj}{{\mathsf{j}}}
\nc{\sk}{{\mathsf{k}}}
\nc{\sm}{{\mathsf{m}}}
\nc{\sn}{{\mathsf{n}}}
\nc{\so}{{\mathsf{o}}}
\nc{\sq}{{\mathsf{q}}}
\nc{\sr}{{\mathsf{r}}}
\nc{\st}{{\mathsf{t}}}
\nc{\su}{{\mathsf{u}}}
\nc{\sv}{{\mathsf{v}}}
\nc{\sw}{{\mathsf{w}}}
\nc{\sx}{{\mathsf{x}}}
\nc{\sy}{{\mathsf{y}}}
\nc{\sz}{{\mathsf{z}}}

\nc{\oA}{{\overline{A}}}
\nc{\oB}{{\overline{B}}}
\nc{\oC}{{\overline{C}}}
\nc{\oD}{{\overline{D}}}
\nc{\oE}{{\overline{E}}}
\nc{\oF}{{\overline{F}}}
\nc{\oG}{{\overline{G}}}
\nc{\oH}{{\overline{H}}}
\nc{\oI}{{\overline{I}}}
\nc{\oJ}{{\overline{J}}}
\nc{\oK}{{\overline{K}}}
\nc{\oL}{{\overline{L}}}
\nc{\oM}{{\overline{M}}}
\nc{\oN}{{\overline{N}}}
\nc{\oO}{{\overline{O}}}
\nc{\oP}{{\overline{P}}}
\nc{\oQ}{{\overline{Q}}}
\nc{\oR}{{\overline{R}}}
\nc{\oS}{{\overline{S}}}
\nc{\oT}{{\overline{T}}}
\nc{\oU}{{\overline{U}}}
\nc{\oV}{{\overline{V}}}
\nc{\oW}{{\overline{W}}}
\nc{\oX}{{\overline{X}}}
\nc{\oY}{{\overline{Y}}}
\nc{\oZ}{{\overline{Z}}}

\nc{\oa}{{\overline{a}}}
\nc{\ob}{{\overline{b}}}
\nc{\oc}{{\overline{c}}}
\nc{\od}{{\overline{d}}}
\nc{\of}{{\overline{f}}}
\nc{\og}{{\overline{g}}}
\nc{\oh}{{\overline{h}}}
\nc{\oi}{{\overline{i}}}
\nc{\oj}{{\overline{j}}}
\nc{\ok}{{\overline{k}}}
\nc{\ol}{{\overline{l}}}
\nc{\om}{{\overline{m}}}
\nc{\on}{{\overline{n}}}
\nc{\oo}{{\overline{o}}}
\nc{\op}{{\overline{p}}}
\nc{\oq}{{\overline{q}}}
\nc{\os}{{\overline{s}}}
\nc{\ot}{{\overline{t}}}
\nc{\ou}{{\overline{u}}}
\nc{\ov}{{\overline{v}}}
\nc{\ow}{{\overline{w}}}
\nc{\ox}{{\overline{x}}}
\nc{\oy}{{\overline{y}}}
\nc{\oz}{{\overline{z}}}

\nc{\tA}{{\tilde{A}}}
\nc{\tB}{{\tilde{B}}}
\nc{\tC}{{\tilde{C}}}
\nc{\tD}{{\tilde{D}}}
\nc{\tE}{{\tilde{E}}}
\nc{\tF}{{\tilde{F}}}
\nc{\tG}{{\tilde{G}}}
\nc{\tH}{{\tilde{H}}}
\nc{\tI}{{\tilde{I}}}
\nc{\tJ}{{\tilde{J}}}
\nc{\tK}{{\tilde{K}}}
\nc{\tL}{{\tilde{L}}}
\nc{\tM}{{\tilde{M}}}
\nc{\tN}{{\tilde{N}}}
\nc{\tO}{{\tilde{O}}}
\nc{\tP}{{\tilde{P}}}
\nc{\tQ}{{\tilde{Q}}}
\nc{\tR}{{\tilde{R}}}
\nc{\tS}{{\tilde{S}}}
\nc{\tT}{{\tilde{T}}}
\nc{\tU}{{\tilde{U}}}
\nc{\tV}{{\tilde{V}}}
\nc{\tW}{{\tilde{W}}}
\nc{\tX}{{\tilde{X}}}
\nc{\tY}{{\tilde{Y}}}
\nc{\tZ}{{\tilde{Z}}}

\nc{\tfD}{{\tilde{\fD}}}
\nc{\tcA}{{\tilde{\cA}}}
\nc{\tcB}{{\tilde{\cB}}}
\nc{\tcC}{{\tilde{\cC}}}
\nc{\tcD}{{\tilde{\cD}}}
\nc{\tcE}{{\tilde{\cE}}}
\nc{\tcF}{{\tilde{\cF}}}
\nc{\tcM}{{\tilde{\cM}}}
\nc{\tcP}{{\tilde{\cP}}}
\nc{\tcT}{{\tilde{\cT}}}

\nc{\ta}{{\tilde{a}}}
\nc{\tb}{{\tilde{b}}}
\nc{\tc}{{\tilde{c}}}
\nc{\td}{{\tilde{d}}}
\nc{\te}{{\tilde{e}}}
\nc{\tf}{{\tilde{f}}}
\nc{\tg}{{\tilde{g}}}
\nc{\ti}{{\tilde{\imath}}}
\nc{\tj}{{\tilde{j}}}
\nc{\tk}{{\tilde{k}}}
\nc{\tl}{{\tilde{l}}}
\nc{\tm}{{\tilde{m}}}
\nc{\tn}{{\tilde{n}}}
\nc{\tp}{{\tilde{p}}}
\nc{\tq}{{\tilde{q}}}
\nc{\tr}{{\tilde{r}}}
\nc{\ts}{{\tilde{s}}}
\nc{\tu}{{\tilde{u}}}
\nc{\tv}{{\tilde{v}}}
\nc{\tw}{{\tilde{w}}}
\nc{\tx}{{\tilde{x}}}
\nc{\ty}{{\tilde{y}}}
\nc{\tz}{{\tilde{z}}}

\nc{\hA}{{\hat{A}}}
\nc{\hB}{{\hat{B}}}
\nc{\hC}{{\hat{C}}}
\nc{\hD}{{\hat{D}}}
\nc{\hE}{{\hat{E}}}
\nc{\hF}{{\hat{F}}}
\nc{\hG}{{\hat{G}}}
\nc{\hH}{{\hat{H}}}
\nc{\hI}{{\hat{I}}}
\nc{\hJ}{{\hat{J}}}
\nc{\hK}{{\hat{K}}}
\nc{\hL}{{\hat{L}}}
\nc{\hM}{{\hat{M}}}
\nc{\hN}{{\hat{N}}}
\nc{\hO}{{\hat{O}}}
\nc{\hP}{{\hat{P}}}
\nc{\hQ}{{\hat{Q}}}
\nc{\hR}{{\hat{R}}}
\nc{\hS}{{\hat{S}}}
\nc{\hT}{{\hat{T}}}
\nc{\hU}{{\hat{U}}}
\nc{\hV}{{\hat{V}}}
\nc{\hW}{{\hat{W}}}
\nc{\hX}{{\widehat{X}}}
\nc{\hY}{{\hat{Y}}}
\nc{\hZ}{{\hat{Z}}}

\nc{\ha}{{\hat{a}}}
\nc{\hb}{{\hat{b}}}
\nc{\hc}{{\hat{c}}}
\nc{\hd}{{\hat{d}}}
\nc{\he}{{\hat{e}}}
\nc{\hg}{{\hat{g}}}
\nc{\hh}{{\hat{h}}}
\nc{\hi}{{\hat{i}}}
\nc{\hj}{{\hat{j}}}
\nc{\hk}{{\hat{k}}}
\nc{\hl}{{\hat{l}}}
\nc{\hm}{{\hat{m}}}
\nc{\hn}{{\hat{n}}}
\nc{\ho}{{\hat{o}}}
\nc{\hp}{{\hat{p}}}
\nc{\hq}{{\hat{q}}}
\nc{\hr}{{\hat{r}}}
\nc{\hs}{{\hat{s}}}
\nc{\hu}{{\hat{u}}}
\nc{\hv}{{\hat{v}}}
\nc{\hw}{{\hat{w}}}
\nc{\hx}{{\hat{x}}}
\nc{\hy}{{\hat{y}}}
\nc{\hz}{{\hat{z}}}

\nc{\hcC}{{\widehat{\cC}}}
\nc{\hcT}{{\widehat{\cT}}}

\nc{\eps}{\upepsilon}
\nc{\lan}{\big\langle}
\nc{\ran}{\big\rangle}
\nc{\kk}{{\Bbbk}}
\nc{\io}{\upiota}
\nc{\Kr}{\mathsf{Kr}}
\nc{\cKr}{\mathcal{K}\!\mathit{r}}

\nc{\Dm}{\bD^{-}}
\nc{\Db}{\bD^{\mathrm{b}}}
\nc{\Dbc}{\bD^{\mathrm{b}}_{\mathrm{c}}}
\nc{\Dp}{\bD^{\mathrm{perf}}}
\nc{\Dperf}{\bD^{\mathrm{perf}}}
\nc{\Dqc}{\bD_{\mathrm{qc}}}
\nc{\Du}{\bD}
\nc{\Dsing}{\bD^{\mathrm{sg}}}
\nc{\Dg}{\bD^{\mathrm{sg}}}

\def\ol{\overline}

\nc{\Rn}{\rR_{\mathrm{node}}}
\nc{\Cn}{\cC_{\mathrm{node}}}
\nc{\Dfd}[1]{\bD_{\mathrm{fd}}(#1)}

\def\bw#1#2{\textstyle{\bigwedge\hskip-0.9mm^{#1}}\hskip0.2mm{#2}}

\nc{\xrightiso}[1]{ \xrightarrow[{\ \raisebox{0.5ex}[0ex][0ex]{$\sim$}\ }]{#1} }

\nc{\thick}{\mathbf{thick}}

\DeclareMathOperator{\Ext}{\mathrm{Ext}}
\DeclareMathOperator{\cExt}{\mathcal{E}\!\mathit{xt}}

\DeclareMathOperator{\cRHom}{\mathrm{R}\mathcal{H}\mathit{om}}

\DeclareMathOperator{\Coh}{\mathrm{Coh}}

\DeclareMathOperator{\Pic}{\mathrm{Pic}}

\DeclareMathOperator{\rank}{\mathrm{rk}}

\def\wt{\widetilde}

\theoremstyle{plain}

\newtheorem{theorem}{Theorem}[section]

\newtheorem{lemma}[theorem]{Lemma}
\newtheorem{proposition}[theorem]{Proposition}

\newtheorem*{theorem*}{Theorem}

\theoremstyle{definition}

\newtheorem{definition}[theorem]{Definition}
\newtheorem{assumption}[theorem]{Assumption}
\newtheorem{example}[theorem]{Example}

\theoremstyle{remark}

\newtheorem{remark}[theorem]{Remark}

\newenvironment{renumerate}{\begin{enumerate}[label={\textup{(\roman*)}}]}{\end{enumerate}}

\numberwithin{equation}{section} 

\newtheoremstyle{cited}{.5\baselineskip\@plus.2\baselineskip\@minus.2\baselineskip}{.5\baselineskip\@plus.2\baselineskip\@minus.2\baselineskip}{\itshape}{}{\bfseries}{\bfseries .}{5pt plus 1pt minus 1pt}{\thmname{#1}\thmnumber{~#2}\thmnote{ \normalfont#3}}
\theoremstyle{cited}

\newtheorem{citedcor}[theorem]{Corollary}

\title[Codimension one multiple fibers]{Obstructions for
codimension one multiple fibers of Lagrangian and Calabi--Yau fibrations}

\date{10th August 2026}

\author[Y.-J. Kim]{Yoon-Joo Kim}
\address{Department of Mathematics, Columbia University, New York, NY 10027, USA \\and HCMC, Korea Institute for Advanced Study, Seoul, South Korea}
\curraddr{}
\email{yk3029@columbia.edu}

\author[K. Oguiso]{Keiji Oguiso}
\address{Graduate School of Mathematical Sciences
3 Chome-8-1 Komaba, Meguro City,
Tokyo, 153-8914,
Japan \\and National Center for Theoretical Sciences, National Taiwan University, 
Taipei, Taiwan}
\curraddr{}
\email{oguiso@ms.u-tokyo.ac.jp }

\author[E. Shinder]{Evgeny Shinder}
\address{School of Mathematical and Physical Sciences, University of Sheffield, S3 7RH, Sheffield, UK}
\curraddr{}
\email{eugene.shinder@gmail.com}

 

\begin{document}

\begin{abstract}
We prove that compact hyper-Kähler manifolds with a Lagrangian fibration over a projective space have no multiple fibers in codimension one. This has several consequences for the structure of Lagrangian fibrations, including progress on Sawon’s conjecture on general singular fibers, Kamenova--Lu anti-hyperbolicity, and the extension of the Néron model action to a big open subset of the base.

We prove the same result for Calabi--Yau fibrations on simply-connected K-trivial varieties, including elliptic fibrations, with a single exceptional case: an odd-dimensional K-trivial variety with a single fiber of multiplicity 2 over the projective line. This exceptional case is realized by examples of Borisov--Nuer and their generalizations.
\end{abstract}

\maketitle

\tableofcontents

\section{Introduction}

It is well-known that every elliptic K3 surface $f \colon S \to \P^1$, with or without a section, cannot contain any multiple fibers. The goal of this paper is to work out two optimal generalizations of this result to simply-connected K-trivial fibrations of higher dimensions.

Our first generalization concerns Lagrangian fibrations of hyper-K\"ahler manifolds,
which are the most straightforward higher-dimensional analogs of elliptic K3 surfaces.
A hyper-K\"ahler manifold is a compact simply-connected K\"ahler manifold of dimension $2n$ such that $\rH^{0}(X, \Omega_X^2) = \CC \sigma$ for a non-degenerate two-form $\sigma$. 
A Lagrangian fibration $f\colon X \to \P^n$ is a morphism with connected fibers whose general fibers are Lagrangian with respect to $\sigma$. Note that it is conjectured that every nontrivial and non-birational fibration of $X$ is necessarily a Lagrangian fibration in this sense.
One important property of $f$ is its flatness \cite[Corollary 2]{Matsushita-equi}. We refer to the survey \cite{MauriHuybrechts} and references therein for more details on Lagrangian fibrations.

We say a morphism $f\colon X \to B$ has a codimension one multiple fiber if there is a prime divisor $D \subset B$ such that $f^{-1}(D) = mE$ for an integer $m > 1$ and an effective divisor $E \subset X$.
Our first main result is the following.

\begin{theorem}[see Theorem \ref{thm:HK}]\label{thm:lagr-intro}
Let $f\colon X \to \P^n$ be a Lagrangian fibration of a hyper-K\"ahler manifold. Then $f$ has no codimension one multiple fibers.
\end{theorem}

The theorem is in contrast with its local version in \cite{HwangOguiso}; codimension one multiple fibers do arise in Lagrangian fibrations of non-compact holomorphic symplectic manifolds over polydiscs. Furthermore, even for hyper-K\"ahler manifolds, multiple fibers can arise in higher codimensions \cite{Hellmann}.

Theorem~\ref{thm:lagr-intro} is tightly connected to some foundational results in Lagrangian fibrations. It is equivalent to the existence of local sections of $f$ in codimension one by 
\cite[Remarque 1.3]{CampanaMultipleFibers}
or
\cite[Theorem 1.2]{CampanaKamenovaVerbitsky}.
This in turn is related to the study of Tate--Shafarevich twists of Lagrangian fibrations \cite{AbashevaRogov, Kim-neronmodels, Sacca-compactifications, DuttaMatteiShinder}.
In a similar vein, Theorem~\ref{thm:lagr-intro} implies that the non-critical locus of $f$, over a big open subset of $\P^n$, is a torsor under a smooth commutative group scheme \cite[Theorem~5.2]{Kim-neronmodels}.
Lastly, the theorem rules out difficult analysis of the Tate--Shafarevich group in the presence of codimension one multiple fibers in the proof of the boundedness result of projective Lagrangian fibrations, in the particular case when the base is $\P^n$ \cite[\S5]{EFGMS}.

Some results in the literature stated under the validity of Theorem~\ref{thm:lagr-intro} are now unconditional. These include the so-called tcf (torsion and cotorsion free) property of the constructible sheaf $\rR^1 f_*\Z_X$ \cite[Theorem~2.13]{DuttaMatteiShinder}, application to isotrivial Lagrangian fibrations \cite[Theorem~1.7]{KimLazaMartin}, and holomorphic dominability of $X$ by $\CC^{2n}$ \cite[Theorem~4.12]{KamenovaLu}.

It is a conjecture of Sawon that for sufficiently general Lagrangian fibrations their general singular fibers are semistable \cite[Conjecture 1]{Sawon-general}. To this direction, we can obtain the following sharper version of the known results in the literature \cite{Lehn-deformations}, \cite{Sawon-general}, \cite[Lemma~3.13]{Voisin25}.

\begin{citedcor}\label{cor:intro-reduced} 
If $\rho(X) = 1$ and $X$ is non-projective, or $\rho(X) = 2$ and $X$ is projective, then general singular fibers of $f$ are reduced and have type $\mathrm{I}$, $\mathrm{II}$, $\mathrm{III}$, or $\mathrm{IV}$ in the sense of \cite{HO-foliation}.
\end{citedcor}

The new input here is that Theorem \ref{thm:lagr-intro} implies that general singular fibers are necessarily reduced.
Indeed, if a general singular fiber had components with different multiplicities, then the corresponding divisor in $X$ would be reducible, contradicting the assumption on $\rho(X)$; see \cite[Proof of Theorem 2]{Sawon-general} for $\rho = 1$ and \cite[Proof of Lemma 3.13]{Voisin25} for $\rho = 2$. It follows that, using the classification of Hwang and the second named author \cite{HO-foliation, HwangOguiso} (see also \cite{Kim-HO}), general singular fibers of $f$ can only have type I, II, III or IV.

Theorem \ref{thm:lagr-intro} is closely related to the primitivity of the embedding $f^*(\Pic(\P^n)) \subset \Pic(X)$. This is a question initiated by \cite{KamenovaVerbitsky-primitive} and it is interesting in its own right, see e.g. the discussion following Conjecture 1.4 in \cite{DHMV}. The main result of \cite{KamenovaVerbitsky-primitive} together with Theorem \ref{thm:lagr-intro} implies that $f^*(\Pic(\P^n)) \subset \Pic(X)$ is primitive;
see also Remark \ref{rem:primitivity} about how it can be proved directly using our method.
Conversely, the first step of our proof is to show, 
in Proposition \ref{prop:O(1/m)}, that primitivity implies no codimension one multiple fibers.
Theorem \ref{thm:lagr-intro} and this primitivity result were announced independently by Kamenova--Verbitsky \cite{KV-multiple};
their proof is quite different from ours.

\medskip

Our second generalization concerns higher-dimensional Calabi--Yau fibrations; see Conventions and Notation below for precise definitions. The following is the main result in this direction.

\begin{theorem}
[see Theorem \ref{thm:CY}]
\label{thm:CY-intro}
Let $X$ be a simply-connected, smooth projective, and K-trivial variety 
with a (not necessarily flat)
Calabi--Yau fibration 
$f\colon X \to \P^n$ 
of relative dimension $r$.
If $n \ge 2$ or $r$ is odd, then
$f$ has no codimension one multiple fibers.
If $n = 1$ and $r$ is even,
then $f$ has no multiple fibers of multiplicity greater than $2$, and there is at most one fiber with multiplicity $2$.    
\end{theorem}

There exist examples of simply-connected Calabi--Yau threefolds that are K3-fibered over $\P^1$ and that admit a single multiple fiber of multiplicity $2$
\cite{Borisov-Nuer16, Suzuki22}. We give a criterion for the existence of a fibration of this type for an odd-dimensional K-trivial variety $X$:
it suffices that $X$ contains an Enriques--Calabi--Yau divisor $V \subset X$,
see Proposition \ref{prop:counterexample}. This is a generalization of the corresponding result for Calabi--Yau threefolds containing an Enriques surface by Borisov--Nuer \cite[\S8]{Borisov-Nuer16}; their proof applies in this more general situation. Without the simply-connectedness assumption, there can be several multiple fibers, see Remark \ref{ex:K-triv-examples}. When $f$ is an elliptic fibration, Theorem~\ref{thm:CY-intro} holds under weaker conditions: this will be presented in Proposition~\ref{prop:elliptic CY}. 
\medskip

The proofs of Theorems~\ref{thm:lagr-intro} and \ref{thm:CY-intro} use a combination of classical and new techniques:
the cyclic covering trick,
Koll\'ar's vanishings  for higher direct image sheaves,
Horrocks's splitting criterion for vector bundles on $\P^n$, Hilbert polynomials
and perhaps unexpectedly, a lattice norm with respect to the Beilinson basis $[\cO_{\P^n}(-n)], \ldots, [\cO_{\P^n}]$ 
of the Grothendieck ring $\rK_0(\P^n)$ (see \S \ref{sec:Lagrangian}).

More specifically, we can summarize our strategy as follows. Standard deformation techniques for Lagrangian fibrations allow us to assume $f$ is projective in Theorem~\ref{thm:lagr-intro}.
Therefore, to prove the theorems we assume $f$ is a projective fibration admitting a codimension one fiber of multiplicity $m > 1$, and aim to get a contradiction. The first step is to use the cyclic covering trick and simply-connectedness of $X$, and deduce that the class $f^*(\cO_{\P^n}(1)) \in \Pic(X)$ is uniquely divisible by $m$ (see Proposition~\ref{prop:O(1/m)}). Denote the corresponding line bundle by $\cO_X(\frac1{m})$ and consider its higher direct image sheaves $\rR^i f_* \cO_X(\frac1{m})$. By results of Koll\'ar \cite{KollarI, KollarII, Kollar-automorphic}, these higher direct images are locally free and have vanishing higher-degree cohomology groups. The main properties of these sheaves are formulated in Theorem~\ref{thm:mainbundles}.
In particular, as a consequence of Horrocks's criterion, we are able to show that these locally free sheaves are direct sums of line bundles on $\P^n$.
The contradiction in the proof of Theorem~\ref{thm:lagr-intro} is deduced by computing an invariant, which we call the Beilinson norm, of the class $[\rR f_* \cO_X \left(\frac1{m}\right)]$ in $\rK_0(\P^n)$ (see \S \ref{sec:Lagrangian}).
The only property of Lagrangian fibration we use is the equality $\chi(X, \cO_X (\frac{1}{m})^{\otimes j}) = n+1$ for every $j \in \Z$, which is a consequence of the existence of the Beauville--Bogomolov--Fujiki form and Huybrechts--Riemann--Roch formula.
For Theorem~\ref{thm:CY-intro}, we use the integrality constraints on the  Hilbert polynomial of the class $[\rR f_* \cO_X \left(\frac1{m}\right)]$. This will be explicitly computed in \S\ref{sec:CY}. Hilbert polynomial 
obstructions should work
for 
Calabi--Yau fibrations over more general smooth bases.

\medskip

In both Theorem~\ref{thm:lagr-intro} and Theorem~\ref{thm:CY-intro}, our method of proof requires that $X$ is smooth or at least factorial. It is not entirely clear what kind of generalizations to the singular case one can expect. 
Note that a singular Kummer K3 surface already provides a counterexample to 
both Theorems~\ref{thm:lagr-intro} 
and
\ref{thm:CY-intro}
if we allow $X$ to have Gorenstein canonical singularities, see Example \ref{ex:KummerK3}.
We explore some situations where the results still hold for singular varieties $X$ in  \S\ref{sec:singular}.

\medskip

\noindent{\bf Conventions and Notation.}
We work over the field of complex numbers. A variety is an integral separated scheme of finite type over $\CC$.
A variety $X$ is \emph{K-trivial} if it is Gorenstein and $\omega_X \simeq \cO_X$. A \emph{Calabi--Yau manifold} is a smooth projective K-trivial variety $X$ with $\rH^i (X, \cO_X) = 0$ for $0 < i < \dim X$; note that this includes elliptic curves and K3 surfaces. A proper surjective morphism $f : X \to B$ between smooth varieties has a \emph{codimension $1$ multiple fiber} if there exists a prime divisor $D \subset B$ such that $f^{-1}(D) = mE$ for an effective divisor $E$ and $m > 1$. When $f$ is flat, this is equivalent to saying 
that for a general point $b \in D$, the cycle class of $f^{-1}(b)$ is divisible by $m$.
A proper surjective morphism $f : X \to B$ between normal varieties is a \emph{fibration} if it has connected fibers, or equivalently $\cO_B \to f_* \cO_X$ is an isomorphism. 
By the relative dimension of a fibration we mean the dimension of its general fiber $F$.

\medskip

We will mostly work under the following setup throughout the paper.

\begin{assumption} \label{as:CY}
    $X$ is a  simply-connected and K-trivial smooth projective complex variety (hence by the Beauville--Bogomolov decomposition theorem it is a product of projective hyper-K\"ahler manifolds and strict Calabi--Yau manifolds). We further assume that $f \colon X \to B = \P^n$ is a fibration and write $F$ for its general fiber with dimension $r$.
\end{assumption}

In fact, all our results hold for over any algebraically closed
field of characteristic zero and simply-connectedness can be replaced by triviality of the \'etale fundamental group of $X$.

\medskip

\noindent{\bf Acknowledgments.}
We thank Philip Engel and J\'anos Koll\'ar 
for sharing detailed examples and for suggesting alternative approaches to some of our arguments, and Justin Sawon for his comments on the paper.
We also thank
Yajnaseni Dutta,
Daniel Huybrechts,
Stefan Kebekus, 
Dominique Mattei,
Mirko Mauri,
for discussions and interest in our work. 
K.O. is partially supported by JSPS grants 25H00587 and 25K21992.
E.S. was supported by the UKRI Horizon Europe guarantee award `Motivic invariants and birational geometry of simple normal crossing degenerations' EP/Z000955/1.

\section{Computing the higher pushforward sheaves} \label{sec:higher pushforward sheaves}

The following simple result is the starting point for our investigation of codimension $1$ multiple fibers.

\begin{proposition}\label{prop:O(1/m)}
	Let $X$ be a simply-connected normal projective variety and $f\colon X \to B = \P^n$ be a fibration.
    Assume that $D \subset B$ is a reduced divisor and $f^{-1}(D) = mE$ for an effective Cartier divisor $E$
    with $m > 1$. Then
    \begin{enumerate}
        \item $\deg(D)$ and $m$ are coprime.
        \item There exists a unique line bundle $\cO_X(1/m)$ on $X$ satisfying $\cO_X(1/m)^{\otimes m} \simeq f^*(\cO_B(1))$. For every point $b \in B \setminus D$ the restriction of $\cO_X(1/m)$ to the (scheme-theoretic) fiber $f^{-1}(b)$ is trivial.
    \end{enumerate}
\end{proposition}
\begin{proof}
    The first statement is an application of the cyclic covering trick, see e.g. \cite[Proposition 3.1]{HwangOguiso}.
	Indeed, assume that $d \coloneq \deg(D)$ and $m$ have a common prime factor $p$. Consider the degree $p$ cyclic covering  $B' \to B$ branched along $D$ \cite[9.4, 9.5]{Kollar-automorphic}; since we assume $D$ to be reduced, $B'$ is irreducible.
    Let  $X'$ be the normalization of $X \times_B B'$; it is connected because $B'$ and the (general, hence all) fibers of $X' \to B'$ are connected.
    To understand the morphism $X' \to X$, let $f \in \cO_X$ be a local defining equation of $E$. Then $X \times_B B'$ is locally defined by $(t^p = f^m) \subset \AA^1_t \times X$, and the normalization process splits this into $p$ disjoint branches $(t = \zeta^i f^{m/p})$ for $p$-th roots of unity $1, \zeta, \cdots, \zeta^{p-1}$. This shows $X' \to X$ is \'etale. Hence it is a degree $p$ covering, which violates the simply-connectedness of $X$.

    For the existence part in the second statement write $1 = ad + bm$ with $a, b \in \Z$. Let $H \subset X$ be the pullback of a hyperplane in $B$. Since $mE$ is linearly equivalent to $dH$, the line bundle $L = \cO_X(aE + bH)$ satisfies $L^{\otimes m} \simeq \cO_X(adH + bmH) \simeq \cO_X(H)$. Such $L$ is unique because $\Pic(X)$ is torsion-free by  simply-connectedness of $X$. Finally, since both $\cO_X(E)$ and $\cO_X(H)$ restrict trivially to $f^{-1}(b)$ with $b \notin D$, the same holds for $L$.
\end{proof}

\begin{remark}
    In \cite{KamenovaVerbitsky-primitive}, Kamenova and Verbitsky showed that if $f \colon X \to B = \P^n$ is a Lagrangian fibration with no codimension one multiple fibers, then $f^*(\Pic(B)) \subset \Pic(X)$ is a primitive embedding. Proposition~\ref{prop:O(1/m)}(2) is the converse of their result.
\end{remark}

    We will  write
    \begin{equation}\label{eq:def-Ojm}
    \cO_X(j/m) = \cO_X(1/m)^{\otimes j} \qquad\text{for}\quad j \in \Z .
    \end{equation}
    It is a unique line bundle on $X$ with $\cO_X(j/m)^{\otimes m} \simeq f^* \cO_B(j)$.

\begin{proposition}[{\cite{KollarI, KollarII}}]
\label{prop:local freeness of higher direct images}
    Let $X$ be a K-trivial variety with rational singularities, $B$ be a smooth quasi-projective variety, and $f \colon X \to B$ be a projective surjective morphism. Assume that $L$ is a line bundle on $X$ such that $L^{\otimes m}$ is a pullback of a line bundle on $B$ for some $m \ge 1$. Then $\rR^i f_* L$ is a locally free sheaf for every $i$.
\end{proposition}

\begin{proof}
We first explain why the statement is true when $L = \cO_X$.
Let $\pi\colon \wt{X} \to X$ be a resolution of singularities of $X$. We have
\begin{equation}\label{eq:pushf-omega}
\rR \pi_* \cO_{\wt{X}} \simeq \cO_X \simeq \omega_X \simeq \rR \pi_* \omega_{\wt{X}}
\end{equation}
where we used rational singularities and K-triviality of $X$ in the first two isomorphisms and the Grauert--Riemenschneider vanishing theorem \cite[10.16]{Kollar-automorphic} together with $\rR^0 \pi_* \omega_{\wt{X}} \simeq \omega_X$ in the last one. Thus 
if we set $g = f \circ \pi \colon \wt{X} \to B$ and apply the derived pushforward $\rR f_*$ to \eqref{eq:pushf-omega} we get quasi-isomorphisms
\[
\rR f_*\cO_X \simeq \rR g_* \cO_{\wt{X}} \simeq \rR g_* \omega_{\wt{X}}.
\] 
We need to explain why cohomology sheaves of these complexes are locally free.
The key step is that by
\cite[Theorem~2.1(i)]{KollarI} 
the sheaves $\rR^i g_* \omega_{\wt{X}}$ are torsion-free for all $i \ge 0$
and by \cite[Theorem 3.1]{KollarII}
we have a splitting
\[
\rR g_* \omega_{\wt{X}} = \bigoplus_{i = 0}^r
\rR^i g_* \omega_{\wt{X}}[-i].
\]
Note that even though
 \cite[Theorem~2.1]{KollarI} and
 \cite[Theorem~3.1]{KollarII}
 require $\wt{X}$ to be smooth and projective, 
 the projectivity assumption is not needed: that can be deduced using directly by compactifying $X$ and $B$, or using e.g. \cite[Theorem II]{Takegoshi95}.
 
 Local freeness of the coherent sheaves $\rR^i f_* \cO_X \simeq \rR^i g_*\cO_{\wt{X}}$ is shown 
 using an argument from \cite[Proof of Proposition 3.12]{KollarII}
 as follows.
 By Verdier duality 
  for a projective morphism $g\colon \wt{X} \to B$ of relative dimension $r = \dim{\wt{X}} - \dim{B}$, we have
\[
\cRHom_B(\rR g_* \cO_{\wt{X}}, \cO_B)  \simeq
 (\rR g_* \omega_{\wt{X}}) \otimes \omega_B^\vee [r]. 
 \]
 Here the right-hand side complex is a direct sum of shifts of torsion-free coherent sheaves, hence the same holds for the complex in the left-hand side. Since for each $i$, the complex $\cRHom_B(\rR^i g_* \cO_{\wt{X}}[-i], \cO_B)$ is a direct summand 
 of
 $\cRHom_B(\rR g_* \cO_{\wt{X}}, \cO_B)$ again by the splitting,
 this implies that $\rR^i g_* \cO_{\wt{X}}$ are maximal Cohen--Macaulay sheaves: $\cExt^j(\rR^i g_* \cO_{\wt{X}}, \cO_B) = 0$ for $j > 0$, $i \ge 0$ (because these $\cExt$-sheaves are torsion sheaves). 
 On a smooth variety $B$ maximal Cohen--Macaulay sheaves
 are locally free, see e.g. \cite[Lemma 1.20]{Orlov-sing}. This finishes the proof when $L = \cO_X$.
 
 More generally, now assume that $L^{\otimes m} \simeq \cO_X$.
The torsion line bundle $L$ induces a finite \'etale covering $q \colon X' \to X$ and hence a composition 
    \[
    f' \colon X' \xlongrightarrow{q} X \xlongrightarrow{f} B. 
    \]
    By construction, $X'$ is still K-trivial with rational singularities,  we have $q_* \cO_{X'} = \bigoplus_{j=0}^{m-1} L^{\otimes j}$, and the higher pushforward sheaves vanish. Thus $\rR^i f_* L$ are direct summands of $\rR^i f'_* \cO_{X'}$ and we are done by the previous step.

In the general case, 
write $L^{\otimes m} \simeq f^* M$ for a line bundle $M$ on $B$. Since the claim is local on $B$, we can cover $B$ by Zariski open subsets $U \subset B$ trivializing $M$ so that $L$ is torsion on $f^{-1}(U)$.
\end{proof}

\begin{remark}
If $f$ was assumed to be flat, we could have  proved Proposition \ref{prop:local freeness of higher direct images} using \cite{KollarKovacs}. It is remarkable that the result still holds without flatness.
\end{remark}

The two results above combine and yield the following theorem.

\begin{theorem}\label{thm:mainbundles}
    In Assumption~\ref{as:CY}, assume further that there exists a codimension $1$ multiple fiber $f^{-1}(D) = mE$ with $m > 1$. Then for all $0 \le i \le r$ and $j \in \Z$, the higher pushforwards of line bundles \eqref{eq:def-Ojm}
    \[ \Omega^i_j := \rR^i f_* \cO_X(j/m) \]
    are locally free sheaves of rank $h^i(F, \cO_F)$, where $F$ is a general fiber of $f$. They satisfy the relations
    \begin{equation}\label{eq:duality}
        \Omega^i_{j+m} \simeq \Omega^i_j(1)
        \quad\text{and}\quad
        \Omega^{r-i}_{-j} \simeq (\Omega^i_j)^\vee \otimes \omega_B.
    \end{equation}
    Furthermore, if $0 < j < m$, then $\Omega^i_j$ is a sum of line bundles of degrees contained in the range $[-n, 0]$.
\end{theorem}
\begin{proof}
The local freeness of the sheaves $\Omega^i_j$ is a consequence of Proposition~\ref{prop:local freeness of higher direct images}. The two formulas in \eqref{eq:duality} follow immediately from the projection formula and Verdier duality, respectively. The rank of $\Omega_j^i$ can be computed over a general point $b \in B \setminus D$ using flat base change.

It only remains to prove the final statement. Since we have $\omega_X \simeq \cO_X$, the vanishing theorem in \cite[Corollary~10.15.2]{Kollar-automorphic} in our situation reads
\begin{equation}\label{eq:kollar}
h^p(B, \Omega^i_j) = 0 \quad\text{for}\quad p > 0, \ j > 0, \ 0 \le i \le r .
\end{equation}
By Serre duality, this translates into
\[ h^p(B, \Omega^i_j) = 
h^{n-p}(B, \Omega^{r-i}_{-j}) =
0 \quad\text{for}\quad p < n, \ j < 0, \ 0 \le i \le r .\]
In particular, for every $j \not\equiv 0 \pmod{m}$ and $k \in \Z$, we have
\[
h^p(B, \Omega^i_j(k)) = h^p(B, \Omega^i_{j+km})  = 0 \quad\text{for}\quad 0 < p < n , \ 0 \le i \le r .
\]
This shows $\Omega^i_j$ is a direct sum of line bundles by Horrocks's theorem (e.g., \cite[Theorem~2.3.1]{Okonek-Schneider-Spindler-VBonPn}).

The degrees of the line bundles can be computed as follows. Let $\cO(w)$ be a direct summand of $\Omega^i_j$ with $0 < j < m$. Then \eqref{eq:kollar} implies $w \ge -n$. The duality in \eqref{eq:duality} shows $(\Omega^i_j)^\vee(-n) \simeq \Omega^{r-i}_{m-j}$, so again \eqref{eq:kollar}, using Serre duality, implies $-w-n \ge -n$, or equivalently $w \le 0$.
\end{proof}

Note that if $f$ is a Lagrangian fibration, then $\Omega^i_0 \simeq \Omega^i_B$ by \cite{Matsushita05}. This explains our notation in the theorem.

\begin{remark}
If $X$ is smooth, but $B$ singular, the sheaves $\Omega^i_j$ are reflexive, and an appropriate version of \eqref{eq:duality} still holds. These sheaves were used by Ou \cite{Ou} to prove that if $X$ is a smooth hyper-K\"ahler fourfold, then the base $B$ of any Lagrangian fibration $f\colon X \to B$ is factorial; later on Huybrechts--Xu deduced that $B \simeq \P^2$ \cite{HuybrechtsXu}. 
It is generally conjectured that the base $B$ of a Lagrangian fibration is always isomorphic to $\P^n$ -- this is known as soon as $B$ is smooth by a theorem of Hwang \cite{Hwang}.
For our applications we rely on the final statement in Theorem \ref{thm:mainbundles} where it is essential to assume in advance that $B \simeq \P^n$.
\end{remark}

\begin{example}\label{ex:K3}
Let $f\colon S \to \P^1$ be an elliptic K3 surface. Theorem~\ref{thm:mainbundles} immediately recovers the well-known fact that $f$ has no multiple fibers (e.g., see \cite[Proposition 11.1.6]{Huybrechts-K3book}). Indeed, if there is a fiber of multiplicity $m > 1$, then for all $0 < j < m$ we have
\[
\Omega_j^0 \simeq \cO(-v_j), \quad \Omega_j^1 \simeq \cO(-w_j)
\]
for some $v_j, w_j \in \{0, 1\}$. Therefore, the Euler characteristic of $\cO_S (j/m)$ is bounded above as
\[
\chi_S(\cO(j/m))
= \chi_{\P^1}(\cO(-v_j)) - \chi_{\P^1}(\cO(-w_j))  \le 1.
\]
This contradicts the Riemann--Roch computation $\chi_S(\cO(j/m)) = 2$ for all $j \in \Z$.
\end{example}

In the following two sections, we provide generalizations of the arguments of this simple example to higher-dimensional Lagrangian fibrations and Calabi--Yau fibrations, respectively.

\section{Codimension \texorpdfstring{$1$}{1} multiple fibers in Lagrangian fibrations} \label{sec:Lagrangian}

To generalize the method of Example \ref{ex:K3} to higher-dimensional Lagrangian fibrations we will need to work with linear combinations of equivalence classes of coherent sheaves on the base $\P^n$.
We start by recalling some basic results about the Grothendieck group $\rK_0(\P^n)$ of coherent sheaves on $\P^n$. This group is defined as a free abelian group generated by isomorphism classes $[\cF]$ of coherent sheaves $\cF \in \Coh(\P^n)$ modulo relations $[\cF] = [\cF'] + [\cF'']$ for 
every short exact sequence $0 \to \cF' \to \cF \to \cF'' \to 0$. Every bounded complex $\cF^\bullet$ of coherent sheaves on $\P^n$ has a well-defined class $[\cF^\bullet] := \sum_{i \in \Z} (-1)^i [\cH^i(\cF^\bullet)] \in \rK_0(\P^n)$.

Using the Beilinson full exceptional collection, 
e.g., \cite[Corollary 8.29]{Huybrechts-FM}, we obtain
\begin{equation}
\label{eq:beilinson}
\rK_0(\P^n) = \bigoplus_{i=0}^n \Z \cdot [\cO_{\P^n}(-i)].
\end{equation}
Taking into account the ring structure on $\rK_0(\P^n)$ induced by the derived tensor product we can write
\[
\rK_0(\P^n) = \bigoplus_{i=0}^n \Z \zeta^i \simeq \Z[\zeta] / (1-\zeta)^{n+1} \quad\text{where}\quad \zeta = [\cO_{\P^n}(-1)] .
\]
Note that
\begin{equation}\label{eq:1-zeta}
1 - \zeta = [\cO_{\P^n}] - [\cO_{\P^n}(-1)] = [\cO_{\P^{n-1}}]
\end{equation}
is the class of the structure sheaf of a hyperplane $\P^{n-1} \subset \P^n$.
The relation $(1-\zeta)^{n+1} = 0$ follows from the projective bundle formula in K-theory (see e.g. \cite[Example 3.8.1, Theorem 3.9]{Panin}) and
reflects the geometric fact that the intersection of $n+1$ general hyperplanes is empty. 

We define the \emph{Beilinson norm} of a class $\alpha = \sum_{i=0}^n a_i \zeta^i \in \rK_0(\P^n)$  by
\[
\|\alpha\| := \sum_{i=0}^n |a_i|.
\]
This gives $\rK_0(\P^n) \simeq \Z^{n+1}$ the structure of a normed lattice in the sense that
for all $\alpha, \beta \in \rK_0(\P^n)$, $c \in \Z$
\begin{equation}
\label{eq:length-definition}
\|\alpha + \beta \| \le \|\alpha\| + \|\beta\| ,\qquad \|c \cdot \alpha \| = |c|\cdot \|\alpha\| .
\end{equation}

\begin{lemma} \label{lem:norm}
For any $0 \le k \le n$ consider the class $[\cO_{\P^{n-k}}] \in \rK_0(\P^n)$ of the structure sheaf of a linear subspace of codimension $k$.
We have 
\begin{equation}
\label{eq:Opt}
[\cO_{\P^{n-k}}] = (1-\zeta)^k \quad\text{and}\quad \|[\cO_{\P^{n-k}}]\| = 2^k.
\end{equation}
\end{lemma}
\begin{proof}
Using binomial expansion the second equality follows from the first. To compute the class
$[\cO_{\P^{n-k}}]$ we note that
if $i \colon \P^\ell \to \P^{\ell+1}$ is a linear embedding then by projection formula for every $\alpha \in \rK_0(\P^{\ell+1})$
\[
i_*(\rL i^*(\alpha)) =
\alpha \cdot [i_*\cO_{\P^{\ell}}]
=
\alpha \cdot (1-[\cO_{\P^{\ell+1}}(-1)]).
\]
This allows computing $[\cO_{\P^{n-k}}]$ by induction. Alternatively, we can resolve the structure sheaf
$\cO_{\P^{n-k}}$ using the Koszul resolution and  obtain the binomial expansion of $(1-\zeta)^k$ directly.
\end{proof}

Recall that the Euler characteristic $\chi_{\P^n}(-)$ is additive for short exact sequences, hence it descends to a group homomorphism $\chi\colon \rK_0(\P^n) \to \Z$. This homomorphism allows characterizing the class of the structure sheaf of a point in the following way.
For a class $\alpha \in \rK_0(\P^n)$ we write $\alpha(k)$ for $\alpha \cdot [\cO_{\P^n}(k)]$.

\begin{lemma}\label{lem:Opt}
Given $\alpha \in \rK_0(\P^n)$  and $c \in \Z$,
the following are equivalent:
\begin{renumerate}
\item $\alpha = c \cdot [\cO_{p}]$.
\item $\chi(\alpha(k)) = c$ for every $k \in \Z$.
\item $\chi(\alpha(k)) = c$ for $n+1$ consecutive values of $k \in \Z$.
\end{renumerate}
If these conditions are satisfied, then $\|\alpha\| = |c| \cdot 2^n$.
\end{lemma}

\begin{proof}
We use
the (nonsymmetric) Euler pairing on $\rK_0(\P^n)$ defined by $\chi(\alpha, \beta) := \chi_{\P^n}(\alpha^\vee \cdot \beta)$.
Here $(-)^\vee$ is induced by taking derived dual $\cRHom(-, \cO_{\P^n})$.
For coherent sheaves $\cF$ and $\cG$ we have
\[
\chi([\cF], [\cG]) = \chi_{\P^n}(\cRHom(\cF, \cG)) = \sum_{i=0}^n (-1)^i \dim(\Ext^i(\cF, \cG)).
\]
This pairing is unimodular, in particular nondegenerate, because in the basis \eqref{eq:beilinson} it is given by an upper-triangular matrix with diagonal elements equal to $1$. 

Is is clear that (i) implies (ii), and (ii) implies (iii). Let us assume (iii). The class
\[
\alpha - c \cdot [\cO_{p}] \in \rK_0(\P^n)
\]
is $\chi$-orthogonal to $n+1$ consecutive classes $[\cO(-k)]$. Since these classes generate $\rK_0(\P^n)$ and $\chi$ is nondegenerate, we have $\alpha - c\cdot[\cO_{p}] = 0$.

The last statement follows from (i) using
\eqref{eq:length-definition} and \eqref{eq:Opt}.
\end{proof}

\begin{remark}
The proof of Lemma \ref{lem:Opt} 
shows that the group homomorphism
$\rK_0(\P^n) \to \Z^{n+1}$
which sends $\alpha$
to $[\chi_{\P^n}(\alpha(-n)), \ldots, \chi_{\P^n}(\alpha(-1)), \chi_{\P^n}(\alpha)]$
is an isomorphism.
In particular, a class $\alpha \in \rK_0(\P^n))$ is fully determined by its Hilbert polynomial, cf \cite[Exercise III.5.4]{Hartshorne}. However, keeping track of the classes in $\rK_0(\P^n)$ is better, as for example the Beilinson norm of a class is not a natural invariant of its Hilbert polynomial.    
\end{remark}

We are now ready to prove our first main result.

\begin{theorem} \label{thm:HK}
Let $f \colon X \to B = \P^n$ be a Lagrangian fibration of a  hyper-K\"ahler manifold. Then \ $f$ has no codimension $1$ multiple fibers.
\end{theorem}
\begin{proof}
Let us first reduce the claim to the case when $X$ is projective. Let $\mathfrak X \to \mathfrak B$ be the degenerate twistor deformation of $f$, a flat family of Lagrangian fibered hyper-K\"ahler manifolds $f_t \colon X_t \to B$ parametrized by $t \in \CC$ such that $f_0 = f$ \cite[Theorem~1.1]{Soldatenkov-Verbitsky}. The fiber $f_t^{-1}(b)$ is biholomorphic to $f^{-1}(b)$ for every $t \in \CC$ and $b \in B$ \cite[Theorem~1.10]{Verbitsky15} \cite[Theorem 1.2]{AbashevaRogov}. 
Therefore, we may prove the theorem for any $f_t \colon X_t \to B$ instead. Since there exists $t \in \CC$ such that $X_t$ is projective \cite[Proposition~26.6]{Huybrechts-HK}, 
we can simply assume $X$ is a projective hyper-K\"ahler manifold and use the algebraic discussions in the previous section.

Assume on the contrary that there exists a prime divisor $D \subset B$ with $f^{-1}(D) = mE$ and $m > 1$. Take the line bundle $\cO_X(1/m)$ in Proposition~\ref{prop:O(1/m)} and the class
\[
\alpha \coloneq [\rR f_* \cO_X(1/m)] = \sum_{i=0}^n (-1)^i [\Omega_1^i] \in \rK_0(\P^n)
\]
where we use notation from Theorem \ref{thm:mainbundles}.
We claim $\alpha = (n+1) [\cO_p]$. By Lemma~\ref{lem:Opt}, it is enough to compute
for each $k \in \Z$
\begin{equation}\label{eq:chi-np1}
\chi_{\P^n} (\alpha \cdot [\cO_B(k)]) = \chi_X \big(\cO_X(1/m) \otimes f^* \cO_B(k) \big) = \chi_X \big(\cO_X(\tfrac{1+km}{m}) \big) = n+1 .
\end{equation}
Note that the last equality follows from the Huybrechts--Riemann--Roch theorem \cite[Corollary~23.18]{Huybrechts-HK} that the Euler characteristic $\chi_X(L)$ of a line bundle $L$ is a degree $n$ polynomial on its Beauville--Bogomolov--Fujiki value $q(L) \in \Z$ with constant term ${n+1} = \chi_X(\cO_X)$ \cite[Proposition 3]{Beauville-ch1}; since $q(\cO_X(\frac{1+km}{m})) = 0$, the last equality in \eqref{eq:chi-np1} follows.
Now Lemma~\ref{lem:norm} implies
\[
\|\alpha\| = (n+1) \cdot 2^n.
\]
On the other hand, Theorem~\ref{thm:mainbundles} shows 
that the class $[\Omega^i_1]$ is a sum of exactly $\rank(\Omega^i_1)$ basis elements, therefore
$\|[\Omega^i_1]\| = \rank(\Omega^i_1) = \binom{n}{i}$ for every $0 \le i \le n$, since the general fiber of $F$ is an
$n$-dimensional abelian variety.
Thus we obtain using the triangle inequality for the Beilinson norm that
\[
\|\alpha\| = \left\| \, \sum_{i=0}^n (-1)^i [\Omega^i_1] \, \right\| \le \sum_{i=0}^n \|[\Omega^i_1]\| = 
\sum_{i=0}^n \binom{n}{i} = 
2^n,
\]
which contradicts the previous computation.
\end{proof}

\begin{remark}
\label{rem:primitivity}
Theorem \ref{thm:HK}, together with the main result of \cite{KamenovaVerbitsky-primitive} implies
that the embedding 
\[
f^*(\Pic(B)) = \Z\cdot[f^*(\cO_B(1))]
\subset \Pic(X)
\]
is primitive. 
Alternatively, let us sketch how to modify the proof of Theorem \ref{thm:HK} to prove this primitivity result directly.
Assume that $f^*(\cO_B(1)) \simeq \cO_X(1/m)^{\otimes m}$, for some 
$\cO_X(1/m) \in \Pic(X)$ and $m \ge 2$. For the proofs of Theorems \ref{thm:mainbundles} and \ref{thm:HK} to go through and give a contradiction the only thing to check is that the line bundle $\cL \in \Pic(F)$ obtained by restricting $\cO_X(1/m)$ to the general fiber $F$ of $f$ is trivial.
By construction, $\cL$ is an $m$-torsion line bundle, and
if it was nontrivial, then $\rH^i(F, \cL) = 0$ for all $0 \le i \le n$ -- this is well-known for any $\cL \in \Pic^0(F)$ on any abelian variety $F$, but when $\cL$ is torsion this can be computed  immediately using the cyclic covering associated to $\cL$. This in turn forces $\Omega^i_1 = 0$, hence
$\chi_X(\cO_X(1/m)) = \sum_{i=0}^n (-1)^i \chi_B(\Omega_1^i) = 0$ which again contradicts the Huybrecths--Riemann--Roch computation
$\chi_X(\cO_X(1/m)) = n+1$.
\end{remark}

\section{Codimension \texorpdfstring{$1$}{1} multiple fibers in Calabi--Yau fibrations}\label{sec:CY}

Let us return back to Assumption~\ref{as:CY}. We recall that, contrary to the previous section, the fibration $f \colon X \to B = \P^n$ may no longer be flat or equidimensional anymore.

A general fiber $F$ of $f$ is necessarily a K-trivial variety (i.e., $\omega_F \simeq \cO_F$) by adjunction formula. The goal of this section is to study the case when $F$ is moreover a Calabi--Yau manifold --
see Conventions and notation in the Introduction for definitions. 
In particular, both elliptic curves and K3 surfaces are Calabi--Yau manifolds in the above sense and the results in this section will cover elliptic and K3 surface fibrations but not higher-dimensional abelian fibrations.

\subsection{Calabi--Yau fibrations}

\begin{theorem} \label{thm:CY}
Let $f \colon X \to B = \P^n$ be as in Assumption~\ref{as:CY}, and assume further that its general fiber $F$ is a Calabi--Yau manifold.
\begin{enumerate}
    \item If $n \ge 2$ or $r$ is odd, then $f$ has no codimension $1$ multiple fibers.
    \item If $n = 1$ and $r$ is even, then $f$ has no fibers with multiplicity greater than $2$. Moreover, there may be at most single fiber with multiplicity $2$.
\end{enumerate}
\end{theorem}

We emphasize that $f$ is not assumed to be flat. We will see in \S \ref{sec:counterexample} that this theorem is optimal, namely that multiplicity $2$ fibers do exist. The rest of this subsection is devoted to the proof of Theorem~\ref{thm:CY}.
For a  projective variety $X$, any class $\alpha \in \rK_0(X)$
and a line bundle $L \in \Pic(X)$ we can consider the Hilbert function
$p_{\alpha, L}(k) = \chi_X(\alpha \cdot L^{\otimes k})$, for $k \in \Z$. This function is always a polynomial in $k$ with rational coefficients; see \cite[(2.5.3)]{EGAIII-1} when $\alpha = [\cF]$ is a class of coherent sheaf -- the general case follows since $p_{\alpha, L}$ is additive in $\alpha$.
Note that following \cite{EGAIII-1}, in the proof of Theorem \ref{thm:CY}, we will still call $p_{\alpha, L}$ the Hilbert polynomial, even when $L$ is not ample. 

\begin{lemma}\label{lem:strictCY-chi}
Under the assumptions of Theorem \ref{thm:CY}, the Hilbert polynomial $p(k) = \chi_B(\rR f_* \cO_X \otimes \cO_B(k))$ is given by
\[
p(x) = 
\frac{(x+1)(x+2)\cdots(x+n) + (-1)^{r} (x-1)(x-2)\cdots(x-n)}{n!} ,
\]
which is a polynomial with nonnegative coefficients. We have $p\left(\frac12\right) > 1$ unless $n = 1$ and $r$ is even.
\end{lemma}
\begin{proof}
Since the general fiber $F$ is Calabi--Yau,
the first part of the proof of
Theorem~\ref{thm:mainbundles} (that does not use the existence of codimension one multiple fibers)
shows $\rR^if_* \cO_X = 0$ for $0 < i < r$. As we assume $f$ to be a fibration, $f_* \cO_X = \cO_B$, so $\rR^r f_* \cO_X = \omega_B$ follows for the same reason.
Let us consider the Hilbert polynomial $\chi_n(x) := \chi_{\P^n}(\cO(x)) = \frac{(x+1)\cdots (x+n)}{n!}$.
By the projection formula, we can compute $\rR^i f_*\cO_X(k)$ and  deduce
\[
p(k) = \chi_{\P^n}(\cO_{\P^n}(k)) + (-1)^r \chi_{\P^n}(\omega_{\P^n}(k)) = \chi_{n}(k) + (-1)^{n+r} \chi_{n}(-k)
\]
which gives the desired expression for $p(x)$.
It is clear from the formula that $p(x)$ has nonnegative coefficients. 
For $n > 1$ we can bound $p\left(\frac12\right)$ as follows:
\[
p\left(\tfrac12\right) = 
\chi_{n}\left(\tfrac12\right) \pm \chi_{n}\left(-\tfrac12\right) \ge 
\chi_{n}\left(\tfrac12\right) - \chi_{n}\left(-\tfrac12\right)
= \chi_{{n-1}}\left(\tfrac12\right)
> \chi_{{n-1}}(0) = 1.
\]
If $n = 1$, then $p\left(\frac12\right)$ equals $2$ when $r$ is odd (and it equals $1$ when $r$ is even).
\end{proof}

\begin{proof}[Proof of Theorem \ref{thm:CY}]
Assume that $f$ has a codimension one multiple fiber of multiplicity $m > 1$.
Consider another Hilbert polynomial $h(j) := \chi_B \big( \rR f_* (\cO_X(j/m)) \big) = \chi_X \big( \cO_X(j/m) \big)$. Since $\cO_X(j/m) = \cO_X(1/m)^{\otimes j}$, this is also polynomial function in $j$ with rational coefficients. Notice $h(mk) = p(k)$ for every $k \in \Z$ by definition. It follows that $p(x) = h(mx)$ as polynomials, which implies that for all $j \in \Z$
\[
p\big( \tfrac{j}{m} \big) = h(j) = \chi_X( \cO_X(j/m)) \in \Z.
\]
This gives a very strong integrality constraint on $p(x)$.

Let us first assume $n \ge 2$.
Since $p(x)$ is a nonconstant polynomial with nonnegative coefficients it is strictly increasing for $x \ge 0$ and we have
\begin{equation}\label{eq:g-chain}
0 \le p(0) < p\left(\tfrac1{m}\right)
< p\left(\tfrac2{m}\right) < \ldots < 
p\left(\tfrac{m-1}{m}\right) < p(1) .
\end{equation}
On the other hand, for every $0 < j < m$
by  Theorem~\ref{thm:mainbundles}, 
we have $\Omega_j^i = 0$ for $0 < i < r$,
the line bundles $\Omega_j^0$ and  $\Omega_j^r$ have degrees in the interval $[-n, 0]$
hence
\begin{equation}\label{eq:value-1}
0 < p\big( \tfrac{j}{m} \big) = \chi(\cO_X(j/m)) = h^0(\P^n, \Omega^0_j) + (-1)^r h^0(\P^n, \Omega^r_j) \le 1.
\end{equation}
Here the last inequality, when $r$ is even, is justified as follows: we can not have $\Omega_j^0 \simeq \Omega_j^r \simeq \cO_{\P^n}$ because otherwise by duality \eqref{eq:duality} we would get
 $\Omega_{m-j}^0 \simeq \Omega_{m-j}^r \simeq \cO_{\P^n}(-n)$ and $p(\tfrac{m-j}{m}) = 0$ which is impossible.

It follows from \eqref{eq:g-chain} 
and \eqref{eq:value-1}
that $m = 2$ and $p\left(\frac{1}{2}\right) = 1$
however by Lemma \ref{lem:strictCY-chi} this is not possible as we assumed $n \ge 2$. 

If $n = 1$, and $r$ is odd, then $p(x) = 2$
in which case there are no multiple fibers possible by \eqref{eq:value-1}.
Finally, if $n = 1$ and $r$ is even then $p(x) = 2x$ and $p\left(\frac{1}{2}\right) = 1$, so we can have $m = 2$.
In that case, if there exist two such multiple fibers over distinct points $b, b' \in \P^1$, then this violates Proposition~\ref{prop:O(1/m)} applied to a degree $2$ reduced divisor $D = b + b'$. Hence there exists at most one multiple fiber with multiplicity $2$.
\end{proof}

\begin{remark}\label{rem:primitivity-CY}
With the proof of Theorem \ref{thm:CY}, one can also investigate the primitivity
of the embedding
 $f^*(\Pic(B)) \subset \Pic(X)$
for a Calabi--Yau fibration. In addition to assumptions of Theorem \ref{thm:CY}, assume that the general fiber $F$ of $f$ is simply-connected.
Then arguing like in Remark \ref{rem:primitivity}, one can prove that in case (1), this embedding is primitive and in case (2) the index of $f^*(\Pic(B))$ in its saturation is at most two.
\end{remark}

\begin{remark}\label{ex:K-triv-examples}
    Both K-triviality and simply-connectedness are essential assumptions in Theorem \ref{thm:CY}.
    On the one hand, there exist simply connected elliptic surfaces that have multiple fibers but are not K-trivial (e.g. Halphen pencils or Dolgachev surfaces).
    
    On the other hand there exist 
    non-simply connected Calabi--Yau threefolds $X$ 
    with a fibration $f\colon X \to \P^1$ with general fiber a K3 surface, with more than one multiple fiber. Such varieties can be found among so-called Calabi--Yau threefolds of type K \cite[\S 2K]{OguisoSakurai}. 
    For a specific example, let us take $Y = S \times E$ where $S$ is a K3 surface which admits an Enriques involution and $E$ an elliptic curve.
    Consider the 
    action of $G = \Z/2$ on $Y$
    given by the Enriques involution on $S$ and by $-1$ on $E$.
    Take $X = S \times E / G$;
    it is a smooth Calabi--Yau threefold.
    Let $f\colon X \to E/G = \P^1$ be induced by the projection. 
    Then smooth fibers of $f$ are isomorphic to $S$ and there are $4$ double Enriques fibers over the fixed points of $G$ on $E$.
\end{remark}

In the next subsection we recall and generalize the fact that $X$ containing an Enriques surface already implies the existence of a Calabi--Yau fibration on $X$ with a multiple fiber \cite{Borisov-Nuer16}.

\subsection{Simply-connected K-trivial counterexamples} \label{sec:counterexample}
Theorem~\ref{thm:CY} did not rule out a single numerical possibility: the case when $n = 1$ and $r = 2d$ is even, in which case there may be one multiple fiber with multiplicity $2$. In this section, we explain how to construct examples realizing this case, thus showing that the theorem is formulated optimally.

\begin{definition}
An \emph{Enriques--Calabi--Yau manifold} is a smooth projective variety $V$ of even dimension $2d$ such that its universal cover $W$ is a Calabi--Yau manifold but  $\omega_V \not \simeq \cO_V$.
\end{definition} 

An Enriques surface is an Enriques--Calabi--Yau manifold of dimension $2$. More generally, the Hilbert scheme of $d$ points on an Enriques surface is an Enriques--Calabi--Yau manifold of dimension $2d$ \cite[Theorem 3.1]{Oguiso-Schroer11}. Enriques--Calabi--Yau manifolds are also called Enriques varieties of index two, see
\cite[Definition 2.1, Proposition 2.1]{BNWS-Enriques}.

\begin{proposition} \label{prop:counterexample}
Let $X$ be a simply-connected and smooth projective K-trivial variety of dimension $2d+1$ containing an Enriques--Calabi--Yau manifold $V \subset X$ of dimension $2d$. Then $X$ admits a fibration $f \colon X \to B = \P^1$ and with exactly one multiple fiber $2V$ and a Calabi--Yau general fiber $F$.
\end{proposition}

We note that there are many simply-connected Calabi--Yau threefolds $X$ containing Enriques surfaces \cite[\S3]{Borisov-Nuer16}, \cite[\S4]{Suzuki22} (the general fiber $F$ in Proposition~\ref{prop:counterexample} in this case will be a K3 surface), showing that the exceptional numerical possibility in Theorem~\ref{thm:CY} indeed arises. Let us prove this proposition through the rest of this subsection.

\medskip

Write $\pi \colon W \to V$ for the universal cover of $V$.
Since $\deg \pi \cdot \chi (\cO_V) = \chi(\cO_W) = 2$, we have $\deg \pi = 2$ and $\chi(\cO_V) = 1$. Because $\pi^* \omega_V = \omega_W \simeq \cO_W$, we conclude from the projection formula that $\omega_V^{\otimes 2} \simeq \cO_V$ and $\pi_*\cO_W \simeq \cO_V \oplus \omega_V$.

\begin{lemma}\label{lem:cohomology of V}
    Let $V$ be an Enriques--Calabi--Yau manifold of dimension $2d$. Then 
    \begin{renumerate}
    \item $h^0(V, \cO_V) = 1$ \ and \ $h^i (V, \cO_V) = 0$ for $i > 0$.
    \item $h^{2d}(V, \omega_V) = 1$ \ and \ $h^i(V, \omega_V) = 0$ for $i < 2d$.
    \end{renumerate}
\end{lemma}
\begin{proof}
    Consider the equality
    \[ \rH^i(W, \cO_W) = \rH^i(V, \pi_*\cO_W) = \rH^i(V, \cO_V) \oplus \rH^i(V, \omega_V) .\]
    For $1 \le i \le 2d-1$, we have $h^i(W, \cO_W) = 0$ and hence $h^i(V, \cO_V) = h^i(V, \omega_V) = 0$. The computations of $h^0$ and $h^{2d}$  follow using the Serre duality.
\end{proof}

\begin{lemma} \label{lem:2V is base point free}
    $h^0(X, \cO_X(2V)) = 2$ and the complete linear system $|2V|$ is base-point-free.
\end{lemma}
\begin{proof}
    Our argument follows \cite[Proposition~8.1]{Borisov-Nuer16}. The adjunction formula shows $\omega_V = \cO_X(V)|_V = \cO_V(V)$. Since $\omega_V^{\otimes 2} \simeq \cO_V$, we have $\cO_V(2V) \simeq \cO_V$. Now take the exact sequence
    \[ 0 \to \cO_X \to \cO_X(V) \to \cO_V(V) = \omega_V \to 0 .\]
    Applying Lemma~\ref{lem:cohomology of V}, we have $h^0(X, \cO_X(V)) = 1$ and $h^1(X, \cO_X(V)) = 0$. Take another exact sequence
    \[ 0 \to \cO_X(V) \to \cO_X(2V) \to \cO_V(2V) = \cO_V \to 0 .\]
    We can then conclude $h^0(X, \cO_X(2V)) = 2$.
    
    The base locus of $|2V|$ must be contained in $V$. In the above we have showed the surjectivity of $H^0(X, \cO_X(2V)) \to H^0(V, \cO_V(2V)) = H^0 (V, \cO_V) = \CC$, so $|2V|$ has no base locus in $V$ as well.
\end{proof}

\begin{proof}[Proof of Proposition~\ref{prop:counterexample}]
    Let $f \colon X \to B = \P^1$ be the morphism associated to the base-point-free linear system $|2V|$ in Lemma~\ref{lem:2V is base point free}. We compute the cohomology of its general fiber $F$. Take the short exact sequence
    \[ 0 \to \cO_V(-V) = \omega_V \to \cO_{2V} \to \cO_V \to 0 .\]
    Lemma~\ref{lem:cohomology of V} implies $h^i(2V, \cO_{2V}) = 0$ for all $1 \le i \le 2d-1$, so by upper-semicontinuity, we have $h^i(F, \cO_F) = 0$ for all $1 \le i \le 2d-1$. Similarly, we have $h^0 (2V, \cO_{2V}) = 1$ and thus $h^0 (F, \cO_F) = 1$. This concludes that $f$ has connected fibers and $F$ is a Calabi--Yau manifold.
\end{proof}

\section{Generalizations to singular total spaces}\label{sec:singular}

Let us discuss some possible generalizations of Theorems~\ref{thm:HK} and \ref{thm:CY}. The following simple example already shows the subtlety to this direction.

\begin{example}\label{ex:KummerK3}
    If $E$ and $F$ are elliptic curves, then the second projection 
    \[
    f\colon S = (E \times F)/\langle -1 \rangle \to \P^1 \simeq F/\langle -1 \rangle
    \]
    has four multiple fibers, so one cannot expect similar results when $X$ has canonical singularities. Let us analyze this example a bit further.
    Let $p \in \P^1$ be an image of any $2$-torsion point in $F$. We have $f^{-1}(p) = 2 L$, where $L \simeq \P^1$ is a Weil divisor on $S$ which is not Cartier. One can compute directly that for the reflexive powers of $L$ we have
    \[
    \chi_S(\cO(L)^{[k]}) = \begin{cases}
    2 & \text{ $k$ even}\\
    1 & \text{ $k$ odd}\\
    \end{cases}
    \]
    hence this Hilbert function is not a polynomial; see \cite[Theorem 9.1]{Reid-canonical} for the general Riemann--Roch theorem for Weil divisors on projective surfaces with du Val singularities. This explains how the arguments of Example \ref{ex:K3} and the proofs of Theorems \ref{thm:HK} and \ref{thm:CY} do not apply.
\end{example}

On the other hand, it is interesting that sometimes singularities do not affect the codimension $1$ multiple fibers statement.
The following result should be well-known to experts, but we included it for convenience of the readers.

\begin{proposition} \label{prop:elliptic CY}
    Let $X$ be a projective variety with terminal singularities such that $K_X$ is base-point-free and Cartier. Let $f \colon X \to B$ be a  fibration whose general fiber $F$ is an elliptic curve. Then $f$ has no codimension $1$ multiple fibers.
\end{proposition}
\begin{proof}
    Write $n = \dim B$ for simplicity. Take $n-1$ sufficiently ample and general divisors $H_1, \cdots, H_{n-1} \subset B$, and consider the intersections $C = \bigcap_{i=1}^{n-1} H_i \subset B$ and $S = \bigcap_{i=1}^{n-1} f^{-1}(H_i) \subset X$. Then $C$ is a smooth curve since $B$ is normal, and $S$ is a smooth surface since $X$ is terminal. Adjunction formula computes
    \[ K_S = (K_X + \sum_{i=1}^{n-1} f^*H_i)|_S .\]
    This shows $K_S$ is base-point-free and hence nef. In particular, $f_C : S \to C$ is a minimal elliptic surface.

    It now suffices to show that $f_C$ has no multiple fibers. By the canonical bundle formula for minimal elliptic surfaces, multiple fibers of $f_C$ arise as fixed components of $|K_S|$ (see Lemma~\ref{lem:canonical bundle formula}). Since $|K_S|$ was base-point-free, $f_C$ has no multiple fibers.
\end{proof}

\begin{lemma} \label{lem:canonical bundle formula}
    Let $f\colon S \to C$ be a minimal elliptic surface (possibly without any sections). If $K_S$ is effective, then the multiple fibers of $f$ arise as fixed components of $|K_S|$.
\end{lemma}
\begin{proof}
    Let $F_i$ be the underlying reduced divisors for the multiple fibers of $f$ with multiplicities $m_i \ge 2$. 
    Since $K_S$ is numerically $f$-trivial, every $D \in |K_S|$ is of the form $D = \sum_i a_i F_i + E$ where $a_i \ge 0$ and $E$ is a vertical effective divisor with no intersections with $F_i$'s. The canonical bundle formula (e.g., \cite[Theorem~V.12.1]{Barth-Hulek-Peters-VanDeVan}) yields
    \[ \cO_{F_i} (a_iF_i) = \cO_{F_i} (D) = \cO_{F_i} (K_S) = \cO_{F_i} ((m_i-1)F_i) .\]
    But $\cO_{F_i}(F_i)$ is a torsion line bundle of order $m_i$ \cite[Lemma~III.8.3]{Barth-Hulek-Peters-VanDeVan}, so this proves $a_i \ge m_i - 1 \ge 1$ as desired.
\end{proof}

An irreducible symplectic variety is a singular generalization of a hyper-K\"ahler manifold defined in \cite[Definition~1.3]{Greb-Guenancia-Kebekus19}. Corollary~13.3 in the same paper proves that every irreducible symplectic variety is simply-connected.
 Lagrangian fibrations in the singular setting were studied by Schwald \cite{Schwald}. 
By \cite[Theorem 3]{Schwald}, Lagrangian fibrations are generically smooth  and the general fiber is an abelian variety even when $X$ is singular.

\begin{proposition}
    Let $f \colon X \to B = \P^n$ be a  Lagrangian fibration from a factorial irreducible symplectic variety $X$. Then $f$ has no codimension $1$ multiple fibers.
\end{proposition}
\begin{proof}
    All results in \S \ref{sec:higher pushforward sheaves} apply when $X$ is simply-connected, has rational singularities, K-trivial, and the divisor $E$ is Cartier. The last holds when $X$ is factorial. The proof of Theorem \ref{thm:HK} goes through without any change because Huybrechts--Riemann--Roch theorem holds for line bundles on $X$ \cite[Corollary~5.16]{Bakker-Lehn22}. Note just like in the smooth case we have $\chi(\cO_X) = n+1$, see e.g. \cite[Proof of Corollary 13.3]{Greb-Guenancia-Kebekus19}, hence $\chi(L) = n+1$ for any $L \in \Pic(X)$ such that $q(L) = 0$.
\end{proof}

\begin{remark}
Similarly, Theorem \ref{thm:CY} and its proof work in the same way if $X$ is assumed to be projective with rational Gorenstein singularities, K-trivial and factorial. 
\end{remark}


\end{document}